\documentclass[english]{article}
\usepackage[LGR,T1]{fontenc}
\usepackage{xcolor}
\usepackage{amsmath}
\usepackage{amsthm}
\usepackage{amssymb}
\usepackage{graphicx}
\usepackage[square,numbers]{natbib}
\usepackage{hyperref}
\makeatletter
\usepackage{geometry}
\theoremstyle{plain}
\newtheorem{thm}{\protect\theoremname}
\theoremstyle{definition}
\newtheorem{defn}[thm]{\protect\definitionname}
\theoremstyle{remark}
\newtheorem*{rem*}{\protect\remarkname}
\theoremstyle{plain}
\newtheorem{prop}[thm]{\protect\propositionname}
\theoremstyle{plain}
\newtheorem{lem}[thm]{\protect\lemmaname}
\theoremstyle{plain}
\newtheorem{cor}[thm]{\protect\corollaryname}

\makeatother

\usepackage{babel}
\providecommand{\corollaryname}{Corollary}
\providecommand{\definitionname}{Definition}
\providecommand{\lemmaname}{Lemma}
\providecommand{\propositionname}{Proposition}
\providecommand{\remarkname}{Remark}
\providecommand{\theoremname}{Theorem}

\begin{document}

\mbox{}\\[10mm]

\begin{center}

{\Large \bf \sc  Representations of Quantum Minimal Surface Algebras \\ via Kac-Moody-theory}\\[5mm]

\vspace{6mm}
\normalsize
{\large  Jens Hoppe${}^{1}$, Ralf K\"ohl${}^2$, and Robin Lautenbacher${}^2$ }

\vspace{10mm}
${}^1${\it             Technische Universitat Braunschweig,  D-38106 Braunschweig Germany\\
}
\vskip 1 em
${}^2${\it  Justus-Liebig-Universit\"at, Mathematisches Institut, Arndtstra\ss{}e 2, D-35392 Gie\ss{}en, Germany}

\vspace{10mm}

\hrule

\vspace{5mm}

 \begin{tabular}{p{14cm}}

We consider epimorphisms from quantum minimal surface algebras onto involutory subalgebras of split--real, simply-laced Kac--Moody Lie algebras and provide examples of affine and finite type. We also provide epimorphisms onto such Kac--Moody Lie algebras themselves, where reality of the construction is important. The results extend to the complex situation.
\end{tabular}

\vspace{6mm}
\hrule
\end{center}

\thispagestyle{empty}

\newpage
\setcounter{page}{1}

%\setcounter{tocdepth}{1}
%\tableofcontents

%\bigskip

\section{Introduction}

That the space we live in may not be the continuum that we naively
perceive was already realized by B.~Riemann who, while laying the ground
for the (smooth) geometries bearing his name, at the same time (literally:
in a remark that he apparently planned to expand, in
his famous Habilitationsschrift \cite{Riemann_habil}) pointed out the possibility that space at small distances may be discrete (nowadays one could add  \textit{fuzzy} \cite{1982PhDT........32H} \cite{Madore_1992} or \textit{non-commutative} \cite{Connes_1994}). Concrete ideas to quantize \textit{space-time}  seem to have first
appeared in \cite{PhysRev.71.38} and \cite{Yang:1947ud}. In 1996 \cite{Hoppe:1996xp},   \cite{Ishibashi_1997}  the equations 
\begin{equation}\label{eq:quantization}
\left[ \left[ X^\mu,X^\nu\right] , X_\nu  \right]=0
\end{equation}
appeared in the context of quantizing (relativistic) minimal surfaces - hence the name Quantum Minimal Surface Algebras (QMSA). 
In the context of Yang Mills theory, eq. (\ref{eq:quantization}) as well as parts of its representation theory was considered in \cite{Nekrasov:2002kc},\cite{Connes:2002ya},
\cite{10.2307/29783224}, referred to as \textit{Yang-Mills algebra}.
An inhomogenuous version of (\ref{eq:quantization}), also studied in relation to non-selfdual Nahm-equations \cite{Corrigan:1984ut} , was named \textit{Discrete Minimal Surface Algebra(s)} (DMSA) in \cite{Arnlind_2010}. 
Here, we provide epimorphisms of both homogeneous and in-homogeneous QMSAs onto involutory subalgebras $\mathfrak{k}(A)$ of Kac-Moody algebras $\mathfrak{g}(A)$. 
These subalgebras were first described in \cite{Berman89} and particularly $\mathfrak{k}\left(E_9\right)$ and $\mathfrak{k}\left(E_{10}\right)$ play an important role in the fermionic sector of supergravity \cite{Damour_2006}, \cite{Buyl_2006} (see also \cite{Damour_2001} and \cite{DHN_2002}). 
We also provide some Kac-Moody-algebras as epimorphic images as well. From \cite{Herscovich_2015} it is known that a homogeneous QMSA (a.k.a. a Yang-Mills-algebra) on at least $4$ generators admits the complex free Lie algebra on two generators as epimorphic image. As a consequence, such a QMSA admits infinitely many simple, complex Kac-Moody algebras as a quotient.  
While such a wealth of epimorphisms is fascinating, the resulting representations have the drawback that they include both hermitian and anti-hermitian operators. This is problematic for use in physics, due to reality constraints in physical theories. 
Even though our epimorphisms also work in the complex case, they are indeed more tailored towards the real situation, allowing real involutory subalgebras and Kac-Moody algebras of split-real type as images of real QMSAs. 

Let us also mention the following “physical” aspects:
Unitary representations of the Poincaré group are central to the description of point particles - 
including relativistically invariant wave equations whose solutions provide representation space(s) and lead to the quantum fields that are used to describe and calculate concrete effects confirming, and predicted from, physical theories such as QCD or the Weinberg-Salam model.
Extended objects clearly should involve infinite extensions of the Poincaré group. 
As the $X^\mu$ in (\ref{eq:quantization}) are quantum analogues of the position coordinates of extended objects it is reasonable and completely natural to find infinite-dimensional Lie-algebras $\mathfrak{L}$ being related to (\ref{eq:quantization}).
While it could well be that specific  $\mathfrak{L}$s will be singled out by additional criteria (not discussed below), one should also take serious indications for a kind of ‘democracy’ between different $\mathfrak{L}$s , via (\ref{eq:quantization}).
One should not forget that not long after \cite{Riemann_habil} W.Killing’s (independent) approach to what became and nowadays is called ‘Lie’-algebras originated in trying to understand the nature of the space we live in. 
One should therefore not be surprised to, one and a half centuries later, find algebras at the heart of understanding geometry. 
While there are nowadays so many different approaches to ‘Quantum‘ Space-Time that it would be preposterous to put a small selection into the list of references, let us hope that the approach we take below will add something to the understanding of the space-time we know. 
Let us also mention that it is probably no coincidence that the kind of ideas and concepts that were discussed right \textit{before} string theory ( spectrum generating algebra, dynamical symmetry group, composite infinite-dimensional field equations) again (see also, e.g. , \cite{Hoppe11}, \cite{Hoppe13})  appear in the present context.

The document consists of two parts. First we collect rather general criteria for the existence of homomorphism in section \ref{sec:General}. We start with Kac-Moody algebras and proceed with their involutory subalgebras. As two of our examples use affine KM-algebras we discuss them separately in section \ref{subsec:affine_subs}. We conclude with a selection of examples in section \ref{sec:examples}.

\vspace{20pt}
\textbf{Acknowledgments}: The work of R.K. and R.L. has received funding from the Deutsche Forschungsgemeinschaft (DFG) via the grant KO 4323/13-2. The work of R.L. has received funding from the Studienstiftung des deutschen Volkes. We would like to thank A. Kleinschmidt, P. Levy, and T. Damour for discussions, helpful comments on a previous draft, and correspondence.

\section{General statements}\label{sec:General}

\subsection{Kac-Moody algebras}

First of all, let us give the definition of a quantum minimal surface algebra
(QMSA):
\begin{defn}
\label{def:QMSA}The quantum minimal surface algebra (QMSA) of rank
$n$, signature $(n-m,m)$ and spectrum $\left(\mu_{1},\dots,\mu_{n}\right)$
is defined as the quotient of the free Lie algebra $\mathfrak{f}$
on $n$ generators $x_{1},\dots,x_{n}$ w.r.t. the ideal that is generated
by the relations 
\begin{equation}\label{eq:spectrum}
\sum_{i=1}^{n-m}\left[x_{i},\left[x_{i},x_{j}\right]\right]-\sum_{i=n-m+1}^{n}\left[x_{i},\left[x_{i},x_{j}\right]\right]=\mu_{j}x_{j}\ \forall\,j=1,\dots,n.
\end{equation}
We will often abbreviate the map $\sum_{i=1}^{n-m}\left[x_{i},\left[x_{i},\cdot\right]\right]-\sum_{i=n-m+1}^{n}\left[x_{i},\left[x_{i},\cdot\right]\right]$
by $\Delta$ (cp. \cite{arnlind2019quantum}). We denote the QMSA by $\mathfrak{Q}_{n-m,m}\left(\mu_{1},\dots,\mu_{n}\right)$.
The base fields we consider are $\mathbb{R}$ and $\mathbb{C}$.
\end{defn}

\begin{rem*}
For $n=2$ both $\mathfrak{Q}_{2,0}(0,0)$ and $\mathfrak{Q}_{1,1}(0,0)$  are isomorphic to the Lie algebra of strictly upper triangular $3\times3$-matrices. For $n>2$ no such isomorphism is known.
\end{rem*}

For any Lie algebra $\mathfrak{g}$ with $n$ elements $y_{1},\dots,y_{n}$
satisfying eq. (\ref{eq:spectrum}) there exists a homomorphism $\mathfrak{Q}_{n-m,m}\left(\mu_{1},\dots,\mu_{n}\right)\rightarrow\mathfrak{g}$
defined by sending $x_{i}\mapsto y_{i}$ for all $i=1,\dots,n$. In
the following we would like this Lie algebra to be a Kac-Moody algebra
$\mathfrak{g}(A)$ to a generalized Cartan matrix $A$ or an involutory
subalgebra of $\mathfrak{g}(A)$.
\begin{defn}
A matrix $A\in\mathbb{Z}^{n\times n}$ is called a generalized Cartan
matrix (GCM) if 
\[
a_{ii}=2,\quad a_{ij}\leq0,\quad a_{ij}=0\,\Leftrightarrow\,a_{ji}=0
\]
for all $i\neq j\in\left\{ 1,\dots,n\right\} $. A GCM $A$ is called
\textit{symmetrizable} if there exists $D\in\mathbb{Q}^{n\times n}$
diagonal and regular as well as $B\in\mathbb{Z}^{n\times}$ symmetric
such that $A=DB$. If $a_{ij}\in\left\{ 0,-1\right\} $ for all $i\neq j$,
one calls $A$ \textit{simply-laced}. To a GCM one associates a generalized
Dynkin diagram $\mathcal{D}\left(A\right)$ by drawing $n$ nodes
with weighted edges according to the $a_{ij}$, where $a_{ij}=0$
corresponds to no edge.
\end{defn}

\begin{defn}
\label{def:KMA}Let $A$ be an invertible, symmetrizable GCM and $\mathbb{K}=\mathbb{R},\mathbb{C}$
and let $\mathfrak{f}$ be the free Lie algebra over $\mathbb{K}$
on generators $e_{1},\dots,e_{n},h_{1},\dots,h_{n},f_{1},\dots,f_{n}$.
Let $\mathfrak{I}$ denote the ideal in $\mathfrak{f}$ that is generated
by the relations 
\[
\left[h_{i},e_{j}\right]=a_{ij}e_{j},\ \left[h_{i},f_{j}\right]=-a_{ij}f_{j},\ \left[h_{i},h_{j}\right]=0,\ \left[e_{i},f_{j}\right]=\delta_{ij}h_{i}
\]
\[
\text{ad}\left(e_{i}\right)^{1-a_{ij}}\left(e_{j}\right)=0,\ \text{ad}\left(f_{i}\right)^{1-a_{ij}}\left(f_{j}\right).
\]
The Kac-Moody algebra $\mathfrak{g}\left(A\right)\left(\mathbb{K}\right)$
is now defined as the quotient Lie algebra $\mathfrak{f}\diagup\mathfrak{I}$.
Set $\mathfrak{n}_{+}:=\left\langle e_{1},\dots,e_{n}\right\rangle $,
$\mathfrak{n}_{-}:=\left\langle f_{1},\dots,f_{n}\right\rangle $
and $\mathfrak{h}:=\text{span}_{\mathbb{K}}\left\{ h_{1},\dots,h_{n}\right\} $
and call $\mathfrak{h}$ the Cartan subalgebra of $\mathfrak{g}\left(A\right)$.
\end{defn}

\begin{rem*}
If $A$ is not of full rank in the above definition, the construction
yields the derived Lie algebra $\mathfrak{g}'\left(A\right):=\left[\mathfrak{g}(A),\mathfrak{g}(A)\right]$
instead of the full Kac-Moody algebra $\mathfrak{g}\left(A\right)$.
The derived Lie algebra differs in its Cartan subalgebra which is
smaller than that of $\mathfrak{g}\left(A\right)$. The above definition
is by now somewhat standard but that it is equivalent to the original
definition for symmetrizable $A$ is known as the Gabber-Kac-theorem
(cp. \cite{bams/1183548298}). Note that $\mathfrak{n}_{\pm}$ and $\mathfrak{h}$ provide
a triangular decomposition of $\mathfrak{g}(A)$ into subalgebras
but not ideals: 
\[
\mathfrak{g}(A)=\mathfrak{n}_{-}\oplus\mathfrak{h}\oplus\mathfrak{n}_{+}.
\]
\end{rem*}
\begin{prop}
\label{prop:epi into n+-}Let $A\in\mathbb{Z}^{n\times n}$ be a simply-laced
GCM of rank $n$ as in def. \ref{def:KMA}. Consider the QMSA $\mathfrak{Q}_{n-m,m}(0,\dots,0)$
on $n$ generators, then there exist epimorphisms 
\[
\mathfrak{\phi}^{\pm}:\mathfrak{Q}_{n-m,m}\rightarrow\mathfrak{n}^{\pm},\quad x_{i}\mapsto\begin{cases}
\phi^{+}\left(x_{i}\right):=e_{i}\\
\phi^{-}\left(x_{i}\right):=f_{i}
\end{cases}.
\]
They are related via the Cartan-Chevalley involution $\omega$ as one has $\phi^{+}=-\omega\circ\phi^{-} $, $\phi^{-}=-\omega\circ \phi^{+}$.
\end{prop}

\begin{proof}
Since $A$ is simply-laced one has the Serre relations: 
\[
\text{ad}\left(e_{i}\right)^{2}\left(e_{j}\right)=0=\text{ad}\left(f_{i}\right)^{2}\left(f_{j}\right)\ \forall\,i,j.
\]
Then $\sum_{i}\varepsilon_{i}\text{ad}\left(x_{i}\right)^{2}\left(x_{j}\right)=0$
with $\varepsilon_{i}=\pm1$ is satisfied under the image of $\phi^{\pm}$
because one has 
\[
\text{ad}\left(\phi^{\pm}\left(x_{i}\right)\right)^{2}\left(\phi^{\pm}\left(x_{j}\right)\right)=0
\]
for each summand. It is a standard fact that, $\mathfrak{n}_{+}\cong\mathfrak{n}_{-}$
 are conjugate under the Cartan-Chevalley involution of $\mathfrak{g}(A)$. As $\omega\left(e_i\right)=-f_i$  and $\omega\left(f_i\right)=-e_i$, one has $\phi^{+}=-\omega\circ\phi^{-} $, $\phi^{-}=-\omega\circ \phi^{+}$.
\end{proof}
There is actually no need for $n$ of $\mathfrak{Q}_{n-m,m}$ and
$A\in\mathbb{Z}^{k\times k}$ to match. If $n<k$ one can always pick
a sub-diagram of $\mathcal{D}(A)$ of rank $n$. If $n>k$ however,
it is always possible to send all but $k$ generators of $\mathfrak{Q}_{n-m,m}$
to $0$. It would be more interesting to have such epimorphisms to
simple or semi-simple Lie algebras as well. For this one needs to answer
the question of how many generators one needs to generate a simple
Lie algebra. Most amazingly, the answer is universal and turns out
to be $2$:
\begin{thm}\label{thm:two_gens_suffice}
(Adaption of \cite{LU1986470}' main statement) Let $A\in\mathbb{Z}^{k\times k}$
be a GCM of rank $l\geq k-2$. Then $\mathfrak{g}\left(A\right)\left(\mathbb{C}\right)$
is generated by two elements $a$ and $h$, where $a=\sum_{i=1}^{k}\left(e_{i}+f_{i}\right)+h_{0}$
and $h\in\mathfrak{h}$ such that $\left\{ \alpha_{i}\left(h\right),-\alpha_{i}\left(h\right)\right\} $
are pairwise different and $h_{0}\in\mathfrak{h}$ is arbitrary.
\end{thm}
However, these two elements do not satisfy the required relations.
\cite{Herscovich_2015} provides an interesting construction, showing that $\mathfrak{Q}_{2n,0}$ can be mapped onto $\mathfrak{f}_{n}(\mathbb{C})$, the free complex Lie algebra on $n$ generators $y_{1},\dots,y_{n}$.
The epimorphism $\mathfrak{Q}_{2n,0}\left(0,\dots,0\right)\rightarrow\mathfrak{f}_{n}\left(\mathbb{C}\right)$ is such that $x_{j}\mapsto y_{j}$, $x_{j+n}\mapsto i\cdot y_{j}$,
where $x_{1},\dots,x_{2n}$ denote the generators of $\mathfrak{Q}_{2n,0}\left(0,\dots,0\right)$.
As furthermore any simple Kac-Moody algebra $\mathfrak{g}\left(A\right)\left(\mathbb{C}\right)$ with $\text{rnk}\left(A\right)\geq\dim(A)-2$ can be generated by
only two elements, this shows that there exist epimorphisms $\mathfrak{Q}_{2n,0}\left(0,\dots,0\right)\left(\mathbb{K}\right)\rightarrow\mathfrak{g}\left(A\right)\left(\mathbb{C}\right)$, where $\mathbb{K}\neq\mathbb{C}$ is allowed as long as $\mathbb{K}$
is a sub-field of $\mathbb{C}$. Representations $\phi$ derived from this epimorphism always include both hermitian and anti-hermitian operators, since $\phi\left(x_{n+j}\right) = i\cdot\phi\left(x_j\right)$. 
Note that one can always \textit{forget}\footnote{In the sense that one can send them to $0$.} about all but $2$ generators, so that the construction also works for an uneven number of generators. Also, mixed signature can be achieved by exchanging $x_{j+n}\mapsto i\cdot y_{j}$ by $x_{j+n}\mapsto i \cdot y_{j}$ if the pair of generators $x_j$ and $x_{j+m}$ have different signature. Via this construction, the fully split QMSA should actually allow an epimorphism $\mathfrak{Q}_{n,n}\left(0,\dots,0\right)\rightarrow\mathfrak{f}_{n}\left(\mathbb{R}\right)$.
However, it would be nice to also have representations in which all generators
of $\mathfrak{Q}_{n-m,m}$ are realized nontrivially and such that
their images are linearly independent. This and satisfying reality-constraints is our main motivation. Since we cannot rely on an epimorphism to a free Lie-algebra we need to generate our Lie algebra of choice differently than in theorem \ref{thm:two_gens_suffice} in order to satisfy the relations of $\mathfrak{Q}_{n-d,d}$. 
\begin{lem}
Let $\mathfrak{g}\left(A\right)$ be a simple, finite-dimensionale, complex or split-real Lie algebra. 
Denote by $e_{\theta}$ an element in the root space $\mathfrak{g}_{\theta}$, where $\theta$ denotes the highest root of $\Delta(A)$. 
Then $e_{\theta},f_{1},\dots,f_{n}$ generate $\mathfrak{g}(A)$, where $n$ is the rank of $A$.
\end{lem}

\begin{proof}
From representation theory one knows that any weight space in a highest
weight representation can be reached by the iterated action of $f_{1},\dots,f_{n}$.
As $e_{\theta}$ is a highest weight vector w.r.t. the adjoint action,
this implies that $e_{1},\dots,e_{n}$ are contained in $\left\langle e_{\theta},f_{1},\dots,f_{n}\right\rangle $.
Since $\left[e_{i},f_{i}\right]=2h_{i}$ all Chevalley generators
are contained in $\left\langle e_{\theta},f_{1},\dots,f_{n}\right\rangle $
which therefore must be all of $\mathfrak{g}\left(A\right)$.
\end{proof}
For KM-algebras of non-spherical but simply-laced type the above argument
does not work because their root system doesn't possess a highest root.
However, given a real root $\beta$ with maximal support (i.e., in
the decomposition $\beta=\sum_{i=1}^{n}k_{i}\alpha_{i}$ into simple
roots, all $k_{i}$ are nonzero) it is possible to reach every simple
root via a finite number of simple Weyl reflections. This can be translated
to an action via Chevalley generators.
\begin{lem}
\label{lem:max support root implies derived subalgebra is generated}Let
$\mathfrak{g}(A)\left(\mathbb{K}\right)$ be a simply-laced KM-algebra
with irreducible GCM $A$ and $0\neq e_{\beta}\in\mathfrak{g}_{\beta}$
, where $\beta$ is a real positive root with maximal support. Then
$\left\langle e_{\beta},f_{1},\dots,f_{n}\right\rangle _{\mathbb{K}}\cong\mathfrak{g}'(A)\left(\mathbb{K}\right)$.
\end{lem}

\begin{proof}
If $A$ is simply-laced the real roots $\Delta^{re}\left(A\right)$
form a single orbit under the action of the Weyl group and therefore
each $\alpha_{i}$ is conjugate to $\beta$. For each $\alpha_{i}$
there exists a minimal word $W(A)\ni\omega=r_{i_{1}}\cdots r_{i_{d}}$
such that $\text{ht}\left(r_{i_{j}}\cdots r_{i_{d}}\beta\right)<\text{ht}\left(r_{i_{j+1}}\cdots r_{i_{d}}\beta\right)$
for all $j=1,\dots,d$. But then it follows that there exist $p_{1},\dots,p_{d}$
such that 
\[
0\neq\text{ad}\left(f_{i_{1}}\right)^{p_{i_{1}}}\text{ad}\left(f_{i_{2}}\right)^{p_{i_{2}}}\cdots\text{ad}\left(f_{i_{d}}\right)^{p_{i_{d}}}e_{\beta}\in\mathfrak{g}_{\alpha_{i}}.
\]
This implies $e_{1},\dots,e_{n}\in\left\langle e_{\beta},f_{1},\dots,f_{n}\right\rangle $
and therefore $\mathfrak{n}^{\pm}$ are contained as well. Now by
definition $\mathfrak{g}'(A) := \left[\mathfrak{g}(A) , \mathfrak{g}(A)\right] $. 
Since $\mathfrak{h}$ is commutative and acts diagonally on $\mathfrak{n}^{\pm}$,
the intersection $\mathfrak{h}\cap\mathfrak{g}'(A)$ contains all
elements of $\mathfrak{h}$ that can be written as a commutator $\left[x,y\right]$
with $x\in\mathfrak{n}^{-}$ and $y\in\mathfrak{n}^{+}$.
\end{proof}
The other piece of information we need is, if second order commutators
among our generating set vanish. For roots $\alpha\in\Delta^{re}(A)$,
$\beta\in\Delta(A)$ arbitrary but different from $\pm\alpha$, there
exist non-negative integers $p$ and $q$ such that $\beta-p\alpha,\beta-(p-1)\alpha,\dots,\beta,\dots,\beta+q\alpha$
are the only roots of the form $\beta+k\alpha$. The set 
\[
S(\alpha,\beta):=\left\{ \beta-p\alpha,\beta-(p-1)\alpha,\dots,\beta,\dots,\beta+q\alpha\right\} \subset\Delta
\]
is called the root string $S(\alpha,\beta)$, or also the $\alpha$-string through $\beta$. The numbers $p,q$ need
to satisfy $\beta\left(\alpha^{\vee}\right)=p-q$, where $\alpha^{\vee}$
denotes the corresponding coroot to $\alpha$.
\begin{prop}
Let $\beta\in\Delta(A)$ and $\alpha\in\Delta^{re}(A)$. Denote the
number of real roots in $S\left(\alpha,\beta\right)$ by $r\left(\alpha,\beta\right)$.
Then the first and last root of $S\left(\alpha,\beta\right)$ are
real and $r\left(\alpha,\beta\right)\in\left\{ 1,2,3,4\right\} $.
One has the following case distinctions:
\end{prop}

\begin{enumerate}
\item $r\left(\alpha,\beta\right)=1$: $S\left(\alpha,\beta\right)=\left\{ \beta\right\} $
and $\beta$ is real.
\item $r\left(\alpha,\beta\right)=2$: The first and the last root of $S\left(\alpha,\beta\right)$
are real, intermediate roots (if they exist) are imaginary.
\item $r\left(\alpha,\beta\right)=3$: $\left|S\left(\alpha,\beta\right)\right|=3$,
all roots in $S\left(\alpha,\beta\right)$ are real and $\alpha,\beta$
span a sub-root system of type $C_{2}$.
\item $r\left(\alpha,\beta\right)=4$: The first two roots and the last
two roots of $S\left(\alpha,\beta\right)$ are real, all intermediate
roots (if they exist) are imaginary. If $S\left(\alpha,\beta\right)$
does not contain imaginary roots, then $\alpha$ and $\beta$ generate
a root system of type $G_{2}$.
\end{enumerate}
\begin{proof}
This particular form is taken from \cite{10.2748/tmj/1178225523} (prop. 1) but according to the authors parts of it appear as an exercise in \cite{Kac_book} and it is also is used implicitly in \cite{10.2748/tmj/1178227928}.
\end{proof}

For $\alpha,\beta\in\Delta^{re}(A)$ where $A$ is simply-laced, the
options are limited to the first two cases. The root string always
has to start and end with a real root and all roots in the middle
are imaginary. As simply-laced root systems of finite type do not
admit imaginary roots this shows 
\begin{equation}
\text{ad}\left(e_{\alpha}\right)^{2}\left(e_{\beta}\right)=0\ \forall\,\alpha\neq-\beta\in\Delta(A),\label{eq:double commutator vanishes in spherical simply-laced algebra}
\end{equation}
where $A$ is simply-laced and of finite type\footnote{This result is also an immediate consequence of the Serre relations and the fact that non-opposite $\alpha,\beta$ form a prenilpotent pair. For simply-laced and finite $A$ this implies that the Serre-relations hold for any pair of prenilpotent roots.}. For affine root systems
the situation is a bit more complicated but still tractable. We will analyze the example $E_9$ later in the examples section.
\begin{prop}
\label{cor:Epi Q-->classical}Let $A\in\mathbb{Z}^{n\times n}$
be a simply-laced GCM of finite type and $\mathfrak{Q}_{p,q}\left(0,\dots,0\right)$
a QMSA on $p+q=n+m$ generators $x_{1},\dots,x_{n+1}$. There exists
an epimorphism $\phi$ from $\mathfrak{Q}_{p,q}\left(0,\dots,0\right)$
to $\mathfrak{g}\left(A\right)$ defined by 
\[
\phi\left(x_{i}\right)=f_{i}\ \forall\,i=1,\dots,n,\ \phi\left(x_{n+1}\right)=e_{\theta},\ \phi\left(x_{n+j}\right)=e_{\beta_j} \forall\,2\leq j \leq m,
\]
where $f_{1},\dots,f_{n}$ denote the Chevalley generators of $\mathfrak{g}\left(A\right)$
that generate $\mathfrak{n}_{-}$, $e_{\theta}$ is the highest
root vector of $\mathfrak{g}\left(A\right)$ and $e_{\beta_j}\in\mathfrak{g}_{\beta_j}$ such that $\beta_j\in \Delta(A)\setminus\{-\theta,\alpha_1,\dots,\alpha_n\}$ and $\beta_i \neq -\beta_j\,\forall\,i\neq j$.
\end{prop}
\begin{rem*}
It is not necessary for the $\beta_i$ to be different.
\end{rem*}
\begin{proof}
Equation (\ref{eq:double commutator vanishes in spherical simply-laced algebra})
implies that $\phi$ extends to a homomorphism and lemma \ref{lem:max support root implies derived subalgebra is generated}
implies that the image is all of $\mathfrak{g}\left(A\right)$, as
$\mathfrak{g}\left(A\right)=\mathfrak{g}'\left(A\right)$ for $A$
of finite type.
\end{proof}

\subsection{Involutory subalgebras}

The involutory subalgebras we consider are those according to Berman's
construction (\cite{Berman89}) where we restrict to those that do not
involve a diagram automorphism. Berman's involutions are built out
of three mutually commuting automorphisms $\eta$, $\gamma$ and $\tau$
of $\mathfrak{g}$, each of degree $2$ and defined by their action
on the $3n$ Chevalley generators via linear extension. Here, $\pi$
denotes a diagram automorphism of $\mathcal{D}(A)$ and $\tau$ is
a field automorphism of $\mathbb{K}$. For $\mathbb{K}=\mathbb{C}$,
$\tau$ can be chosen to be linear or anti-linear which results in
different subalgebras but we restrict ourselves to the linear case.
The automorphisms $\eta$ and $\gamma$ are always assumed to be $\mathbb{K}$-linear.
\begin{equation}
\eta:\,e_{i}\mapsto f_{i},\ f_{i}\mapsto e_{i},\ h_{i}\mapsto-h_{i}\ \forall\,i=1,\dots,n\label{eq:def of eta}
\end{equation}
\begin{equation}
\gamma:\ e_{i}\mapsto e_{\pi(i)},\ f_{i}\mapsto f_{\pi(i)},\ h_{i}\mapsto h_{\pi(i)}\ \forall\,i=1,\dots,n\label{eq:def of gamma}
\end{equation}
\begin{equation}
\tau:\ e_{i}\mapsto\rho_{i}e_{i},\ f_{i}\mapsto\rho_{i}^{-1}f_{i},\ h_{i}\mapsto h_{i}\ \forall\,i=1,\dots,n\label{eq:def of tau}
\end{equation}
with $\rho_{i}\in\left\{ -1,1\right\} $ and $\rho_{i}=1$ if $\pi(i)\neq i$.
According to (\cite{Berman89}) the last condition is sufficient for
$\tau$ and $\gamma$ to commute. Berman provides generators and relations
for the subalgebra $\text{Fix}_{\sigma}\left(\mathfrak{g}\right)$
where $\sigma=\eta\gamma\tau$. The generators are 
\[
x_{i}:=e_{i}+\rho_{i}f_{\pi(i)},\ z_{i}:=h_{i}-h_{\pi(i)},
\]
 so that we do not have any generators of type $z$ if we don't use
diagram automorphisms. In this case we are able to choose the signs
$\rho_{i}$ freely, so that $\text{Fix}_{\sigma}\left(\mathfrak{g}\right)$
is generated by elements of type $e_{i}\pm f_{i}$. Set 
\begin{equation}
X_{i}:=e_{i}-f_{i},\ Y_{i}:=e_{i}+f_{i}.\label{eq:X_i and Y_i}
\end{equation}
Berman provides the relations among the $x_{i}$ and $z_{i}$ in prop.
1.18 of \cite{Berman89} for the case that all $\rho_{i}=1$. If one
carries the $\rho_{i}$ along his computation one obtains modified
relations. In the simply-laced situation one has the following relations (cp. (3) of \cite{hoppe2021commuting}):
\begin{prop}
\label{lem:generators and relations of inv subs}Let $A$ be a simply-laced
GCM of full rank $n$ and $\mathfrak{g}\left(A\right)\left(\mathbb{K}\right)$.
For an involution $\sigma:=\eta\tau$ with $\eta$ as in eq. (\ref{eq:def of eta})
and $\tau$ $\mathbb{K}$-linear and as in eq. (\ref{eq:def of tau}).
Set $Z_{i}:=e_{i}+\rho_{i}f_{i}$ and denote it by $X_{i}$ if $\rho=-1$
and by $Y_{i}$ if $\rho_{i}=+1$. Then $\text{Fix}_{\sigma}\left(\mathfrak{g}\right)$
is isomorphic to the free Lie algebra generated by $Z_{1},\dots,Z_{n}$
over $\mathbb{K}$ modulo the relations
\end{prop}

\begin{equation}
\left[Z_{i},Z_{j}\right]=0\ \text{if }a_{ij}=0\label{eq:def relations 1}
\end{equation}
\begin{equation}
\left[X_{i},\left[X_{i},X_{j}\right]\right]=-X_{j},\ \left[X_{i},\left[X_{i},Y_{j}\right]\right]=-Y_{j}\ \text{if }a_{ij}=-1\label{eq:def relations 2}
\end{equation}
\begin{equation}
\left[Y_{i},\left[Y_{i},X_{j}\right]\right]=+X_{j},\ \left[Y_{i},\left[Y_{i},Y_{j}\right]\right]=+Y_{j}\ \text{if }a_{ij}=-1\label{eq:def relations 3}
\end{equation}

\begin{rem*}
Similar to Satake-Tits-diagrams we depict the choices of $\rho_{i}$
graphically. If $\rho_{i}=-1$, one colors the node $i$ white and
if $\rho_{i}=+1$ one colors it black.
\end{rem*}
\begin{proof}
This is theorem 1.31 of \cite{Berman89}, where one checks that with
$\rho_{i}=-1$ the relations of prop. 1.18 of \cite{Berman89} adjust
as claimed.
\end{proof}
If one wants to consider larger subalgebras by adding both $X_i$ and $Y_i$ to the generating set, it is potentially useful to know the following relations:
\begin{equation}
\left[X_{i},\left[X_{i},Y_{i}\right]\right]=-4Y_{i},\ \left[Y_{i},\left[Y_{i},X_{i}\right]\right]=+4X_{i}.\label{eq:additional relation}
\end{equation}
Describe $\text{Fix}_{\sigma}\left(\mathfrak{g}\right)$ as being
generated by $\chi:=\left\{ Z_{1},\dots,Z_{l}\right\} $ with $Z_{i}=X_{i}$
or $Y_{i}$ so that one can write 
\[
\Delta_{\chi}=\sum_{i=1}^{n-m}\left[Z_{i},\left[Z_{i},\cdot\right]\right]-\sum_{i=n-m+1}^{n}\left[Z_{i},\left[Z_{i},\cdot\right]\right]\ .
\]
For a simply-laced GCM $A$ and a signature $(n,0)$ one arrives
at the following identities for the action of $\Delta$: 
\begin{equation}
\Delta\left(Z_{i}\right)=\left(n_{black}(i)-n_{white}(i)\right)Z_{i},\label{eq:Delta action}
\end{equation}
where $n_{c}(i)$ denotes the number of neighbors of node $i$ with
color $c$.
\begin{prop}
\label{prop:condition for hom}Let $A$ be a simply-laced generalized
Cartan matrix and $\sigma:=\eta\tau$ an involution as defined by
eqs. (\ref{eq:def of eta}) and (\ref{eq:def of tau}) with $\tau$
$\mathbb{K}$-linear. There exists an epimorphism $\mathfrak{Q}_{n,0}\left(\mu_{1},\dots,\mu_{n}\right)\rightarrow\text{Fix}_{\sigma}\left(\mathfrak{g}\left(A\right)\right)$
given by sending the generators $x_{1},\dots,x_{n}$ of $\mathfrak{Q}_{n,0}\left(\mu_{1},\dots,\mu_{n}\right)$
to the Berman generators $Z_{1},\dots,Z_{n}$ (see \ref{eq:X_i and Y_i})
of $\text{Fix}_{\sigma}\left(\mathfrak{g}\left(A\right)\right)$ if
the coloring which defines $\tau$ is such that
\begin{equation}
\mu_{i}=n_{black}(i)-n_{white}(i)\ \forall\,i=1,\dots,n.\label{eq:condition for hom}
\end{equation}
\end{prop}

\begin{proof}
One applies equation (\ref{eq:Delta action}) to check that the relations
(\ref{eq:spectrum}) are satisfied. Thus, the map 
\[
x_{i}\mapsto Z_{i}\ \forall\,i=1,\dots,n
\]
 extends to a homomorphism of Lie algebras. Since the image contains
all Berman generators $Z_{1},\dots,Z_{n}$ this homomorphism is surjective.
\end{proof}
From \cite{HKL15} one knows that $\text{Fix}_{\omega}\left(\mathfrak{g}\left(A\right)\left(\mathbb{R}\right)\right)$
has a finite-dimensional representation ($\omega$ denotes the Chevalley
involution), called the spin-$\frac{1}{2}$ representation. We now
claim that this representation also provides representations for arbitrary
$\rho_{i}=\pm1$.
\begin{prop}
\label{prop:1/2 spin rep}Let $A$ be a simply-laced generalized Dynkin
diagram, $\mathfrak{g}\left(A\right)$ the corresponding split-real
Kac-Moody algebra and consider the involutory subalgebra $\text{Fix}_{\sigma}\left(\mathfrak{g}\left(A\right)\right)$,
where $\sigma=\eta\tau$ with $\eta$ and $\tau$ defined in (\ref{eq:def of eta})
(\ref{eq:def of tau}). Then there exists a f.d. representation $\phi$
of $\text{Fix}_{\sigma}\left(\mathfrak{g}\left(A\right)\right)$ whose
representation matrices of the Berman generators $Z_{1},\dots,Z_{n}$
of $\text{Fix}_{\sigma}\left(\mathfrak{g}\left(A\right)\right)$ satisfy
the following equations:
\begin{eqnarray*}
\phi\left(Z_{i}\right)^{2} & = & \frac{\rho_{i}}{4}Id\\
\left[\phi\left(Z_{i}\right),\phi\left(Z_{j}\right)\right] & = & 0\ \ \text{if }a_{ij}=0\\
\left\{ \phi\left(Z_{i}\right),\phi\left(Z_{j}\right)\right\}  & = & 0\ \ \text{if }a_{ij}=-1,
\end{eqnarray*}
where $\left\{ A,B\right\} :=AB+BA$ denotes the anti-commutator and $\rho_i\in\{\pm1\}$ denotes the sign which determines $Z_i=e_i+\rho_i f_i$.
\end{prop}

\begin{proof}
For $\mathcal{D}\left(A\right)$ simply-laced and of rank $n$, remark
3.7 and theorem 3.9 of \cite{HKL15} provide the existence
of $n$ matrices $A_{1},\dots,A_{n}$ which satisfy 
\begin{eqnarray*}
(i)\quad & A_{i}^{2}=-\frac{1}{4}Id & \text{for all }i=1,\dots,n\\
(ii)\quad & \left[A_{i},A_{j}\right]=0 & \text{if }a_{ij}=0\\
(iii)\quad & \left\{ A_{i},A_{j}\right\} =0 & \text{if }a_{ij}=-1.
\end{eqnarray*}
Now the matrices $B_{i}:=I\cdot A_{i}$, where $I=\sqrt{-1}$ also
satisfy $(ii)$ and $(iii)$ but one has 
\[
B_{i}^{2}=+\frac{1}{4}Id.
\]
The question now is, if 
\[
\rho\left(Z_{i}\right)=\begin{cases}
A_{i} & \text{if }Z_{i}=X_{i}\\
B_{i} & \text{if }Z_{i}=Y_{i}
\end{cases}
\]
extends to a homomorphism of Lie algebras. Towards this, relations
(\ref{eq:def relations 1}) and (\ref{eq:def relations 2}) are satisfied
by linearity in $Y_{i}$. Towards (\ref{eq:def relations 3}) compute
for $a_{ij}=-1$ that
\begin{eqnarray*}
\left[\rho\left(Y_{i}\right),\left[\rho\left(Y_{i}\right),\rho\left(Z_{j}\right)\right]\right] & = & \rho\left(Y_{i}\right)^{2}\rho\left(Z_{j}\right)-2\rho\left(Y_{i}\right)\rho\left(Z_{j}\right)\rho\left(Y_{i}\right)+\rho\left(Z_{j}\right)\rho\left(Y_{i}\right)^{2}\\
 & = & B_{i}^{2}\rho\left(Z_{j}\right)-2B_{i}\rho\left(Z_{j}\right)B_{i}+\rho\left(Z_{j}\right)B_{i}^{2}\\
 & = & \frac{1}{2}\rho\left(Z_{j}\right)+2B_{i}^{2}\rho\left(Z_{j}\right)=\rho\left(Z_{j}\right),
\end{eqnarray*}
where $\rho\left(Z_{j}\right)B_{i}=-B_{i}\rho\left(Z_{j}\right)$
follows from the fact that $\left\{ A_{i},A_{j}\right\} =0$ implies
$\left\{ A_{i},I\cdot A_{j}\right\} =0$ because the anti-commutator
is bilinear. Thus, (\ref{eq:def relations 3}) is satisfied.
\end{proof}
\begin{rem*}
Theorem 3.14 of \cite{HKL15} shows that $\left\langle A_{1},\dots,A_{n}\right\rangle _{\mathbb{K}}$
is compact and therefore this representation maps the $X_{i}=e_{i}-f_{i}$
to anti-hermitian matrices while the $Y_{i}=e_{i}+f_{i}$ are mapped
to hermitian ones.
\end{rem*}

\subsection{Involutory subalgebras of affine Kac-Moody-algebras}\label{subsec:affine_subs}

For $A$ a classical Cartan matrix of rank $n$, the real loop algebra
$\mathfrak{L}\left(\mathfrak{g}\left(A\right)\right)$ is defined
as the $\mathbb{R}$-tensor product of the ring of Laurent-polynomials
$\mathfrak{L}:=\mathbb{R}\left[t,t^{-1}\right]$ with the classical
split-real Lie algebra $\mathfrak{g}\left(A\right)\left(\mathbb{R}\right)$.
Explicitly, 
\begin{equation}
\mathfrak{L}\left(\mathfrak{g}\left(A\right)\right):=\mathfrak{L}\otimes_{\mathbb{R}}\mathfrak{g}\left(A\right)\label{eq:def of loop algebra}
\end{equation}
as vector spaces and the Lie bracket is given via 
\begin{equation}
\left[P\otimes x,Q\otimes y\right]=\left(P\cdot Q\right)\otimes\left[x,y\right],\label{eq:bracket in loop algebra}
\end{equation}
where the bracket on the right hand side is the Lie bracket in $\mathfrak{g}\left(A\right)$
and $\cdot$ denotes the multiplication of Laurent polynomials. Denote
the Chevalley involution on $\mathfrak{g}\left(A\right)$ by $\mathring{\omega}$,
i.e., let $\mathring{\omega}$ denote the involution that is determined
by its action on the Chevalley generators $e_{i},f_{i},h_{i}$ of
$\mathfrak{g}\left(A\right)$: 
\[
\mathring{\omega}\left(e_{i}\right)=-f_{i},\ \mathring{\omega}\left(f_{i}\right)=-e_{i},\ \mathring{\omega}\left(h_{i}\right)=-h_{i}\ \forall\,i=1,\dots,n.
\]
In the terminology of eqs. (\ref{eq:def of eta}), (\ref{eq:def of gamma})
and (\ref{eq:def of tau}) this corresponds to the situation where
$\gamma$ is trivial and $\tau$ is such that each sign $\rho_{i}=-1$.
Denote by $\theta$ the highest root of $\Delta\left(A\right)$ and
pick a normalized vector $e_{\theta}\in\mathfrak{g}_{\theta}$. With
$\mathring{\omega}$ the Cartan involution of $\mathfrak{g}\left(A\right)$
and $e_{-\theta}:=-\mathring{\omega}\left(e_{\theta}\right)$, set
\begin{equation}
e_{0}:=t\otimes e_{-\theta},\ f_{0}:=t^{-1}\otimes e_{\theta},\label{eq:def of e_0 and f_0}
\end{equation}
then 
\[
\mathfrak{L}\left(\mathfrak{g}\left(A\right)\right)=\left\langle e_{0},f_{0},\mathfrak{g}\left(A\right)\right\rangle .
\]
As vector spaces the affine extension $\mathfrak{g}\left(\widetilde{A}\right)$
of $\mathfrak{g}\left(A\right)$ is given as 
\[
\mathfrak{g}\left(\widetilde{A}\right)=\mathfrak{L}\left(\mathfrak{g}\left(A\right)\right)\oplus\mathbb{K}\cdot K\oplus\mathbb{K}\cdot d,
\]
where $K$ and $d$ are the missing elements in the Cartan subalgebra
of $\mathfrak{g}\left(\widetilde{A}\right)$ and the above $e_{0},f_{0}$
coincide with the Chevalley generators of $\mathfrak{g}\left(\widetilde{A}\right)$.
The precise nature of $K$ and $d$ is not important to us, since
we are interested in the following statement:
\begin{lem}
Let $\mathfrak{g}\left(\widetilde{A}\right)$ be the affine extension
of $\mathfrak{g}\left(A\right)$ and let $\sigma$ be an involution
as in eqs.(\ref{eq:def of eta}), (\ref{eq:def of gamma}) and (\ref{eq:def of tau})
but such that $\gamma$ is trivial and $\tau$ is $\mathbb{K}$-linear.
Then $\text{Fix}_{\sigma}\left(\mathfrak{g}\left(\widetilde{A}\right)\right)\subset\mathfrak{L}\left(\mathfrak{g}\left(A\right)\right)$.
\end{lem}

\begin{proof}
If $\gamma$ is trivial then $\sigma\left(h\right)=-h$ for all $h\in\mathfrak{h}\left(\widetilde{A}\right)$,
in particular $\sigma\left(K\right)=-K$ and $\sigma\left(d\right)=-d$.
Thus, these elements are not fixed by $\sigma$ and therefore not
contained in $\text{Fix}_{\sigma}\left(\mathfrak{g}\left(\widetilde{A}\right)\right)$.
\end{proof}
So it suffices to study the loop algebra to understand $\text{Fix}_{\sigma}\left(\mathfrak{g}\left(\widetilde{A}\right)\right)$
and we would like to describe its structure in a little bit more detail.
\begin{lem}
\label{lem:structure of involutory subalgebras}Let $\mathring{\sigma}:\mathfrak{g}\left(A\right)\rightarrow\mathfrak{g}(A)$
be an involution of Berman type with trivial $\gamma$, $\mathbb{K}$-linear
$\tau$ and denote by $\beta_{\pm}:\mathfrak{L}\rightarrow\mathfrak{L}$
the ring involution determined by $t\mapsto\pm t^{-1}$. Then $\sigma:=\beta_{\pm}\otimes\mathring{\sigma}$
is an involution of the loop algebra $\mathfrak{L}\left(\mathfrak{g}(A)\right)$.
It is also of Berman type with $\rho_{0}=\pm\varepsilon$ iff $\mathring{\sigma}\left(e_{\theta}\right)=\varepsilon e_{-\theta}$
(after restriction from $\mathfrak{g}\left(\widetilde{A}\right)$
to $\mathfrak{L}\left(\mathfrak{g}(A)\right)$ ). Set 
\[
\mathfrak{L}^{\pm}:=\left\{ p\in\mathfrak{L}\,\vert\,\beta_{\pm}(p)=\pm p\right\} ,\ \mathfrak{s}^{\pm}:=\left\{ x\in\mathfrak{g}\left(A\right)\,\vert\,\mathring{\sigma}(x)=\pm x\right\} ,
\]
where the $\pm$ signs in the first set are independent of $\beta_{\pm}$,
then 
\begin{equation}
\text{Fix}_{\sigma}\left(\mathfrak{g}\left(\widetilde{A}\right)\right)=\mathfrak{L}^{+}\otimes\mathfrak{s}^{+}\oplus\mathfrak{L}^{-}\otimes\mathfrak{s}^{-}.\label{eq:split of Fix-sub}
\end{equation}
Furthermore, $\mathfrak{L}^{+}\otimes\mathfrak{s}^{+}$ is a subalgebra,
$\mathfrak{L}^{-}\otimes\mathfrak{s}^{-}$ a $\mathfrak{L}^{+}\otimes\mathfrak{s}^{+}$-module
and $\left[\mathfrak{L}^{-}\otimes\mathfrak{s}^{-},\mathfrak{L}^{-}\otimes\mathfrak{s}^{-}\right]$
$\subset\mathfrak{L}^{+}\otimes\mathfrak{s}^{+}$.
\end{lem}

\begin{proof}
First of all $\sigma$ is an involution because of the way the Lie
bracket is defined on $\mathfrak{L}\left(\mathfrak{g}(A)\right)$.
From $\mathring{\sigma}\left(e_{\theta}\right)=\varepsilon e_{-\theta}$
one deduces that 
\[
\sigma\left(f_{0}\right)=\beta_{\pm}\left(t^{-1}\right)\otimes\mathring{\sigma}\left(e_{\theta}\right)=\pm t\otimes\left(\varepsilon e_{-\theta}\right)=\pm\varepsilon e_{0}
\]
and vice versa. The relations $\mathfrak{L}^{\pm}\cdot\mathfrak{L}^{\pm}\subset\mathfrak{L}^{+}$,
$\mathfrak{L}^{\pm}\cdot\mathfrak{L}^{\mp}\subset\mathfrak{L}^{-}$
hold for the $\pm1$-eigenspaces of any ring involution as well as
$\left[\mathfrak{s}^{\pm},\mathfrak{s}^{\pm}\right]\subset\mathfrak{s}^{+}$,
$\left[\mathfrak{s}^{\pm},\mathfrak{s}^{\mp}\right]\subset\mathfrak{s}^{-}$
hold for any Lie algebra involution.
\end{proof}
The representation theory of such Lie algebras is rather unexplored
at the moment. A major obstacle is that the Lie algebra does not have
a graded but only a filtered structure. An easy way to produce at
least some representations is via evaluation maps. The evaluation
map $ev_{a}$ for $0\neq a\in\mathbb{K}$ is defined as
\begin{equation}
ev_{a}:\mathfrak{L}\rightarrow\mathbb{K},\quad P(t)\mapsto P(t=a)\in\mathbb{K}.\label{eq:evaluation map}
\end{equation}
It is a ring homomorphism and any representation of $\mathfrak{g}\left(A\right)$
extends to a representation of $\mathfrak{L}\left(\mathfrak{g}\left(A\right)\right)$
via such evaluation maps. Explicitly, let $\rho:\mathfrak{g}\left(A\right)\rightarrow\text{End}\left(V\right)$
be a representation, then extend this representation to $\mathfrak{L}\left(\mathfrak{g}\left(A\right)\right)$
by setting 
\begin{equation}
\rho_{a}\left(P(t)\otimes x\right):=P(a)\cdot\rho\left(x\right).\label{eq:extension of reps to loop algebra}
\end{equation}
One can also use the structure of (\ref{eq:split of Fix-sub}). Given
a representation $\left(\phi,V\right)$ of $\mathfrak{s}^{+}$ one
can set 
\begin{equation}
\widetilde{\phi}_{a}\left(p\otimes x+q\otimes y\right)=ev_{a}\left(p\right)\cdot\phi(x)\ \forall\,p\in\mathfrak{L}^{+},q\in\mathfrak{L}^{-},x\in\mathfrak{s}^{+},y\in\mathfrak{s}^{-},\label{eq:Dirac spinor rep}
\end{equation}
which also provides a representation. Note that this representation
has a rather large kernel as all elements in $\mathfrak{L}^{+}\otimes\mathfrak{s}^{+}$
that can be written as a commutator of elements in $\mathfrak{L}^{-}\otimes\mathfrak{s}^{-}$
automatically vanishes.

Evaluation maps throw away lots of information because they ignore
most of the loop structure. It is possible to relate such filtered
Lie algebras to graded ones as described in \cite{kleinschmidt2021representations} and
 \cite{Nicolai:2004nv, Kleinschmidt:2006dy} for the case where $\sigma$ is
the Chevalley involution. However, the methods from \cite{kleinschmidt2021representations} should also
work for arbitrary $\mathring{\sigma}$.

%%%%%%%%%%%%%%%%%%%%%%%%%%%%%%%%%%%%%%%%%%%%%%%%%%%

\section{Examples}\label{sec:examples}
We conclude with some examples that may be interesting towards physical applications.

\subsubsection*{The euclidean example $\mathfrak{Q}_{4,0}\left(0,0,0,0\right)\left(\mathbb{R}\right)$
and the Lorentzian $\mathfrak{Q}_{3,1}\left(0,0,0,0\right)\left(\mathbb{R}\right)$}

As a first example consider the affine Kac-Moody algebra $\mathfrak{g}\left(\tilde{\text{A}}_{3}\right)\left(\mathbb{R}\right)$
and its involutory subalgebras where 
\[
\tilde{\text{A}}_{3}=\begin{pmatrix}2 & -1 & 0 & -1\\
-1 & 2 & -1 & 0\\
0 & -1 & 2 & -1\\
-1 & 0 & -1 & 2
\end{pmatrix}.
\]

We will show that $\mathfrak{Q}_{4,0}\left(0,0,0,0\right)\left(\mathbb{R}\right)$
maps to an involutory subalgebra of $\mathfrak{g}:=\mathfrak{g}\left(\tilde{\text{A}}_{3}\right)\left(\mathbb{R}\right)$.
Color the diagram $\mathcal{D}\left(\tilde{\text{A}}_{3}\right)$
such that nodes $1$ and $2$ are white and $3$ and $4$ are black (compare figure \Ref{fig:A_3_tildes} a).
Then with $\kappa:=\eta\tau$ as in eq. (\ref{lem:generators and relations of inv subs})
one has 
\begin{equation}
\text{Fix}_{\kappa}\left(\mathfrak{g}\right)=\left\langle Z_{1}=X_{1},Z_{2}=X_{2},Z_{3}=Y_{3},Z_{4}=Y_{4}\right\rangle _{\mathbb{R}}\label{eq:split sub in affine A3}
\end{equation}
 and one computes with (\ref{eq:Delta action}) that 
\[
\Delta\left(Z_{i}\right)=0\ \forall\,i=1,2,3,4.
\]
We will now show that the abstract involution $\kappa$ can be expressed in terms of an involution on the classical part $\mathfrak{g}\left(A_3\right)(\mathbb{R})\cong \mathfrak{sl}\left(4,\mathbb{R}\right)$ and an involution on the Laurent polynomials. 

\begin{figure}[h]
\label{fig:A_3_tildes}
\begin{centering}
\includegraphics[scale=0.5]{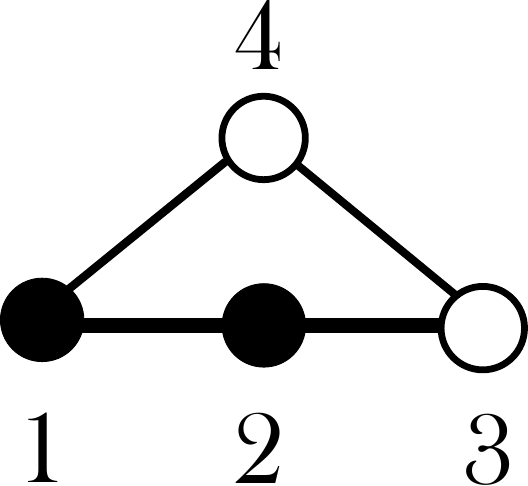}$\quad$ \includegraphics[scale=0.5]{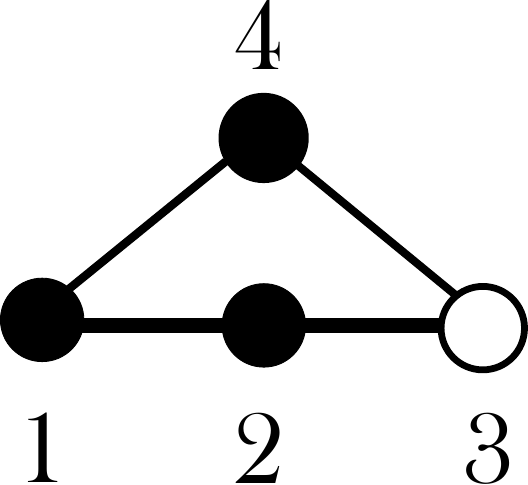}
\par\end{centering}
\caption{Two different colorings of $\widetilde{A}_3$.}
\end{figure}

\begin{prop}
\label{prop:description of Berman sub via loop algebra}The subalgebra
$\text{Fix}_{\kappa}\left(\mathfrak{g}\right)$ of $\mathfrak{g}\left(\widetilde{A}_{3}\right)\left(\mathbb{R}\right)$
is isomorphic to the subalgebra $\text{Fix}_{\sigma}\left(\mathfrak{L}\left(\mathfrak{g}\left(A\right)\right)\right)$
with $\sigma=\beta_{+}\otimes\mathring{\sigma}$ where $\beta_{+}$
is as in lemma \ref{lem:structure of involutory subalgebras} and
$\mathring{\sigma}$ is the involution of Berman type on $\mathfrak{g}\left(A_{3}\right)\left(\mathbb{R}\right)$
determined by 
\[
\mathring{\sigma}\left(e_{1}\right)=-f_{1},\ \mathring{\sigma}\left(e_{2}\right)=-f_{2},\ \mathring{\sigma}\left(e_{3}\right)=+f_{3}.
\]
With $\mathfrak{s}^{\pm}:=\left\{ x\in\mathfrak{g}\left(A_{3}\right)\left(\mathbb{R}\right)\,\vert\,\mathring{\sigma}\left(x\right)=\pm x\right\} $
and $\mathfrak{L}^{\pm}=\left\{ P\in\mathbb{R}\left[t,t^{-1}\right]\,\vert\,P\left(t^{-1}\right)=P(t)\right\} $
one obtains the following structure:
\[
\text{Fix}_{\kappa}\left(\mathfrak{g}\right)\cong\mathfrak{L}^{+}\otimes\mathfrak{s}^{+}\oplus\mathfrak{L}^{-}\otimes\mathfrak{s}^{-}.
\]
There exists an epimorphism from $\mathfrak{Q}_{4,0}\left(0,0,0,0\right)\left(\mathbb{R}\right)$
to $\text{Fix}_{\kappa}\left(\mathfrak{g}\right)$.
\end{prop}

\begin{proof}
The involution $\kappa$ that describes $\text{Fix}_{\kappa}\left(\mathfrak{g}\right)$
is of Berman type and its action on $e_{i},f_{i}$ for $i=1,2,3$
coincides with the above action of $\mathring{\sigma}$. The action
of $\beta_{+}\otimes\mathring{\sigma}$ is already determined by $\mathring{\sigma}$
and we need to show that it coincides with $\kappa$. The highest
root of $A_{3}$ is $\theta=\alpha_{1}+\alpha_{2}+\alpha_{3}$ and
a suitable basis vector is 
\[
e_{\theta}:=\left[e_{1},\left[e_{2},e_{3}\right]\right]\ \Rightarrow\ e_{-\theta}=-\mathring{\omega}\left(e_{\theta}\right)=-(-1)^{3}\left[f_{1},\left[f_{2},f_{3}\right]\right]=\left[f_{1},\left[f_{2},f_{3}\right]\right].
\]
Now 
\[
\mathring{\sigma}\left(e_{-\theta}\right)=(-1)^{2}\left[e_{1},\left[e_{2},e_{3}\right]\right]=e_{\theta},
\]
so that 
\[
\sigma\left(e_{0}\right)=f_{0}=\kappa\left(e_{0}\right).
\]
Hence, lemma \ref{lem:structure of involutory subalgebras} can be
applied. Application of prop. \ref{prop:condition for hom} shows
that there exists an epimorphism from $\mathfrak{Q}_{4,0}\left(0,0,0,0\right)\left(\mathbb{R}\right)$
to $\text{Fix}_{\kappa}\left(\mathfrak{g}\right)$ because the diagram
is suitably colored.
\end{proof}
\begin{prop}
\label{prop:hom and structure for Q3-1}Denote by $\kappa=\eta\tau$
the $\mathbb{K}$-linear involution of Berman type on $\mathfrak{g}\left(\widetilde{A}_{3}\right)\left(\mathbb{R}\right)$
that is given by the signs (compare figure \Ref{fig:A_3_tildes} b) $\rho_{3}=+1$, $\rho_{0}=\rho_{1}=\rho_{2}=-1$
in (\ref{eq:def of tau}) and trivial diagram automorphism. Then there
exists an epimorphism of Lie algebras $\mathfrak{Q}_{3,1}\left(0,0,0,0\right)\left(\mathbb{R}\right)\rightarrow\mathfrak{S}:=\text{Fix}_{\kappa}\left(\mathfrak{g}\left(\widetilde{A}_{3}\right)\left(\mathbb{R}\right)\right)$,
given by sending the generators of $\mathfrak{Q}_{3,1}$ to the Berman
generators of $\mathfrak{S}$. Denote by $\mathring{\sigma}$ the
restriction of $\kappa$ to the naturally contained subalgebra $\mathfrak{g}\left(A_{3}\right)\left(\mathbb{R}\right)$
and recall from lemma \ref{lem:structure of involutory subalgebras}
that $\beta_{-}$ denotes the ring automorphism on the Laurent polynomials
$\mathfrak{L}$ that is defined by $t\mapsto-t^{-1}$. Then, $\sigma:=\beta_{-}\otimes\mathring{\sigma}$
coincides with $\kappa$ on the loop algebra $\mathfrak{L}\left(\mathfrak{g}\left(A_{3}\right)\right)\subset\mathfrak{g}\left(\widetilde{A}_{3}\right)$
and $\text{Fix}_{\kappa}\left(\mathfrak{g}\left(\widetilde{A}_{3}\right)\left(\mathbb{R}\right)\right)=\text{Fix}_{\sigma}\left(\mathfrak{L}\left(\mathfrak{g}\left(A_{3}\right)\right)\right)$.
One has that 
\[
\mathfrak{S}=\mathfrak{L}^{+}\otimes\mathfrak{s}^{+}\oplus\mathfrak{L}^{-}\otimes\mathfrak{s}^{-},
\]
where $\mathfrak{L}^{+}=\mathbb{R}\cdot1\oplus\mathbb{R}\left\{ t^{n}+(-1)^{n}t^{-n},\ n=1,\dots\right\} $
, $\mathfrak{L}^{-}:=\mathbb{R}\left\{ t^{n}-(-1)^{n}t^{-n},\ n=1,\dots\right\} $,
and $\mathfrak{s}^{\pm}:=\left\{ x\in\mathfrak{L}\left(\mathfrak{g}\left(A_{3}\right)\right)\,\vert\,\mathring{\sigma}(x)=\pm x\right\} $
denote the respective $\pm1$ eigenspaces of the involved involutions.
\end{prop}

\begin{proof}
It is stated in lemma \ref{lem:structure of involutory subalgebras},
that $\sigma$ is of Berman type, where the sign $\rho_{0}$ for the
affine node is given as $\rho_{0}=-\varepsilon$ with $\mathring{\sigma}\left(e_{\theta}\right)=\varepsilon e_{-\theta}$.
In the proof of proposition \ref{prop:description of Berman sub via loop algebra}
we computed that $\mathring{\sigma}\left(e_{\theta}\right)=e_{-\theta}$
and therefore $\rho_{0}=-1$. This shows that $\sigma=\kappa$ and
so the only thing left to show is that the Berman generators $Z_{0},Z_{1},Z_{2},Z_{3}$
of $\mathfrak{S}$ satisfy 
\[
\left(-\text{ad}\left(Z_{0}\right)^{2}+\sum_{i=1}^{3}\text{ad}\left(Z_{i}\right)^{2}\right)\left(Z_{\mu}\right)=0\ \forall\,\mu=0,\dots,3.
\]
 Now $Z_{0}=X_{0}=e_{0}-f_{0}$ implies that $\text{ad}\left(Z_{0}\right)^{2}\left(Z_{1}\right)=-Z_{1}$
and $\text{ad}\left(Z_{0}\right)^{2}\left(Z_{3}\right)=-Z_{3}$. One
computes 
\[
\left(-\text{ad}\left(Z_{0}\right)^{2}+\sum_{i=1}^{3}\text{ad}\left(Z_{i}\right)^{2}\right)\left(Z_{0}\right)=\left(\text{ad}\left(Z_{1}\right)^{2}+\text{ad}\left(Z_{3}\right)^{2}\right)\left(Z_{0}\right)=\left(-1+1\right)\left(Z_{0}\right)
\]
\[
\left(-\text{ad}\left(Z_{0}\right)^{2}+\sum_{i=1}^{3}\text{ad}\left(Z_{i}\right)^{2}\right)\left(Z_{1}\right)=\left(-\text{ad}\left(Z_{0}\right)^{2}+\text{ad}\left(Z_{2}\right)^{2}\right)\left(Z_{1}\right)=\left(1-1\right)\left(Z_{1}\right)
\]
\[
\left(-\text{ad}\left(Z_{0}\right)^{2}+\sum_{i=1}^{3}\text{ad}\left(Z_{i}\right)^{2}\right)\left(Z_{2}\right)=\left(\text{ad}\left(Z_{1}\right)^{2}+\text{ad}\left(Z_{3}\right)^{2}\right)\left(Z_{2}\right)=\left(-1+1\right)\left(Z_{2}\right)
\]
\[
\left(-\text{ad}\left(Z_{0}\right)^{2}+\sum_{i=1}^{3}\text{ad}\left(Z_{i}\right)^{2}\right)\left(Z_{3}\right)=\left(-\text{ad}\left(Z_{0}\right)^{2}+\text{ad}\left(Z_{2}\right)^{2}\right)\left(Z_{3}\right)=\left(1-1\right)\left(Z_{3}\right).
\]
\end{proof}
Now $\mathfrak{s}^{+}\cong\mathfrak{so}\left(1,3\right)$ so that
all representations of $\mathfrak{so}\left(1,3\right)$ induce a representation
of $\mathfrak{Q}_{4,0}\left(0,0,0,0\right)\left(\mathbb{R}\right)$.
This includes for instance the Dirac spinors. Furthermore, infinite
dimensional unitary representations of $SO\left(1,3\right)$ provide
anti-hermitian representations of $\mathfrak{Q}_{4,0}\left(0,0,0,0\right)\left(\mathbb{R}\right)$
via this construction.

\begin{figure}[h]
\label{fig:A_n_tilde}
\begin{centering}
\includegraphics[scale=0.5]{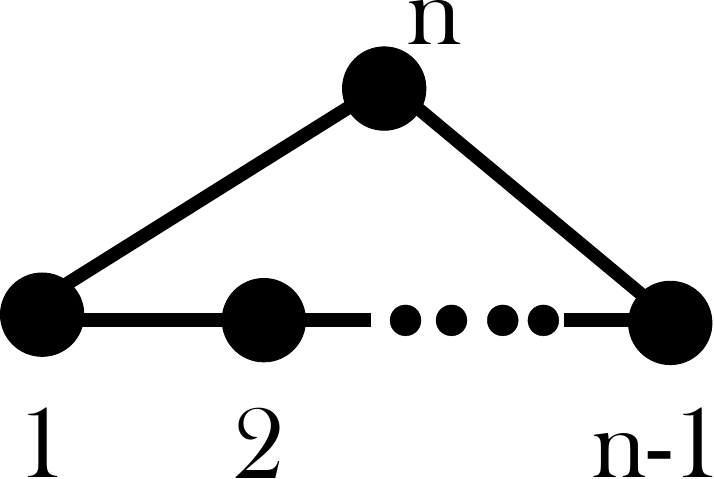}
\par\end{centering}
\caption{The generalized Dynkin diagram to the $\widetilde{A}_{n-1}$-series.}
\end{figure}

The above examples can be extended to $\mathfrak{Q}_{4n}\left(0,\dots,0\right)\left(\mathbb{R}\right)$
and $\mathfrak{Q}_{4n-1,1}\left(0,\dots,0\right)\left(\mathbb{R}\right)$
for any $n\geq1$, as the diagram then is also suitably colored. Then
an (anti-)hermitian representation of $\mathfrak{L}\left(\mathfrak{g}\left(A_{4n-1}\right)\right)$
induces one for $\text{Fix}_{\sigma}\left(\mathfrak{g}\left(\widetilde{A}_{4n-1}\right)\right)$
and therefore for $\mathfrak{Q}_{4n}\left(0,\dots,0\right)$ and $\mathfrak{Q}_{4n-1,1}\left(0,\dots,0\right)\left(\mathbb{R}\right)$.
By the evaluation map $ev_{a}$ any (anti-)hermitian representation
of $\mathfrak{g}\left(A_{4n-1}\right)$ extends to such a representation
of $\mathfrak{L}\left(\mathfrak{g}\left(A_{n-1}\right)\right)$ if
$a$ is real. Note that we considered the split-real form of $\mathfrak{g}\left(A_{4n-1}\right)$
which is $\mathfrak{sl}\left(4n,\mathbb{R}\right)$ and therefore
any such representation will be infinite-dimensional, as $SL\left(4n,\mathbb{R}\right)$
does not admit finite-dimensional unitary representations.

\subsubsection*{A homomorphism $\mathfrak{Q}_{n,0}\left(-2,\dots,-2\right)\left(\mathbb{R}\right)\twoheadrightarrow\mathfrak{k}\left(\widetilde{A}_{n-1}\right)\left(\mathbb{R}\right)$
and anti-hermitian representations}

Each node in Dynkin diagram of $\widetilde{A}_{n-1}$ has exactly
two neighbors (compare figure \ref{fig:A_n_tilde}). Taking the maximal compact subalgebra $\mathfrak{k}\left(\widetilde{A}_{n-1}\right)$
corresponds to coloring all nodes black and therefore there exists
an epimorphism from $\mathfrak{Q}_{n}\left(-2,\dots,-2\right)$ to
$\mathfrak{k}\left(\widetilde{A}_{n-1}\right)$ according to prop.
\ref{prop:condition for hom}.
\begin{cor}
There exists an epimorphism from $\mathfrak{Q}_{n}\left(-2,\dots,-2\right)$
to $\mathfrak{k}\left(\widetilde{A}_{n-1}\right)$ which is defined
by sending the generators $x_{1},\dots,x_{n}$ of $\mathfrak{Q}_{n}\left(-2,\dots,-2\right)$
to the Berman generators $X_{1}:=e_{1}-f_{1}$, $\dots$, $X_{n}=e_{n}-f_{n}$
of $\mathfrak{k}\left(\widetilde{A}_{n-1}\right)$.
\end{cor}

Now what anti-hermitian representations does this imply for $\mathfrak{Q}_{n,0}\left(-2,\dots,-2\right)\left(\mathbb{R}\right)$?
First of all, any highest weight representation of $\mathfrak{g}\left(\widetilde{A}_{n-1}\right)$
possesses a contravariant form w.r.t. which the $X_{i}$ are skew-adjoint
(see chapter 11 of \cite{Kac_book}, the result also applies for $\mathfrak{g}$
split-real and $\mathfrak{k}$ as above). But again we can exploit
that $\mathfrak{k}\left(\widetilde{A}_{n-1}\right)$ is contained
in the loop algebra $\mathfrak{L}\left(\mathfrak{g}\left(A_{n-1}\right)\right)$.
Any element of $\mathfrak{k}\left(\widetilde{A}_{n-1}\right)$ can
be written as a linear combination\footnote{We provided this description of $\mathfrak{k}\left(\widetilde{A}_{n-1}\right)\subset\mathfrak{L}\left(\mathfrak{g}\left(A_{n-1}\right)\right)$
in \cite{kleinschmidt2021representations}.} of 
\[
t^{n}\otimes E_{\alpha}-t^{-n}\otimes E_{-\alpha}\ ,\alpha\in\Delta\left(A_{n-1}\right),\quad\left(t^{n}-t^{-n}\right)\otimes h,\ h\in\mathring{\mathfrak{h}}
\]
where $E_{-\alpha}:=-\mathring{\omega}\left(E_{\alpha}\right)$ for
all $\alpha\in\Delta\left(A_{n-1}\right)$ and $\mathring{\mathfrak{h}}$
denotes the Cartan subalgebra of $\mathfrak{g}\left(A_{n-1}\right)$.
Let $\left(\phi,V\right)$ be a highest weight representation of $\mathfrak{g}\left(A_{n-1}\right)$
then the representation matrices can be chosen such that 
\[
\phi\left(E_{\alpha}\right)^{\dagger}=\phi\left(E_{-\alpha}\right),\ \phi\left(h\right)^{\dagger}=\phi(h)\ \forall\ h\in\mathring{\mathfrak{h}},
\]
so that upon identification of $1\otimes E_{\alpha}$ with $E_{\alpha}$
the elements $1\otimes E_{\alpha}-1\otimes E_{-\alpha}$ act by skew-adjoint
operators. Which evaluation maps $ev_{a}$ preserve this? One has
for $a\in\mathbb{R}$
\[
\phi_{a}\left(t^{n}\otimes E_{\alpha}-t^{-n}\otimes E_{-\alpha}\right)=a^{n}\cdot\phi\left(E_{\alpha}\right)-a^{-n}\phi\left(E_{-\alpha}\right),\quad\phi_{a}\left(\left(t^{n}-t^{-n}\right)\otimes h\right)=\left(a^{n}-a^{-n}\right)\cdot\phi\left(h\right)
\]
and so 
\[
\phi_{a}\left(t^{n}\otimes E_{\alpha}-t^{-n}\otimes E_{-\alpha}\right)^{\dagger}=a^{n}\cdot\phi\left(E_{-\alpha}\right)-a^{-n}\phi\left(E_{\alpha}\right)
\]
which is skew-adjoint if $a^{n}=a^{-n}$ for all $n$ which only leaves
$a=\pm1$. This also implies $ev_{a}\left(t^{n}-t^{-n}\right)=0$
and therefore the $\left(t^{n}-t^{-n}\right)\otimes h$ act trivially.
\begin{prop}
Any highest weight representation $\left(\phi,V\right)$ of $\mathfrak{g}\left(A_{n-1}\right)$
induces an anti-hermitian representation $\left(\phi_{\pm},V\right)$
of $\mathfrak{k}\left(\widetilde{A}_{n-1}\right)\subset\mathfrak{L}\left(\mathfrak{g}\left(A_{n-1}\right)\right)$
given by combining $\phi$ with the evaluation map at $\pm1$. By
the epimorphism $\mathfrak{Q}_{n,0}\left(-2,\dots,-2\right)\left(\mathbb{R}\right)\rightarrow\mathfrak{k}\left(\widetilde{A}_{n-1}\right)\left(\mathbb{R}\right)$
this result extends to $\mathfrak{Q}_{n,0}\left(-2,\dots,-2\right)\left(\mathbb{R}\right)$.
\end{prop}

Also, the representations described in \cite{kleinschmidt2021representations} provide representations that do not stem from representations of $\mathfrak{g}\left(\widetilde{A}_{n-1}\right)\left(\mathbb{R}\right)$ or $\mathfrak{g}\left(A_{n-1}\right)\left(\mathbb{R}\right)$.

\subsubsection*{A homomorphism $\mathfrak{Q}_{10-d,d}\left(0,\dots,0\right)\twoheadrightarrow\mathfrak{g}'\left(E_{9}\right)$}

Different to the previous examples we now try to find a set of $10$ elements $y_{1},\dots,y_{10}$
which satisfy $\text{ad}\left(y_{i}\right)^{2}\left(y_{j}\right)=0$
already without any summation. Then, any result we obtain will hold
for arbitrary signature $(10-d,d)$.
\begin{cor}[{to prop. \ref{prop:epi into n+-}}]There exists an epimorphism $\mathfrak{Q}_{10-d,d}\left(0,\dots,0\right)\rightarrow\mathfrak{n}^{+}\left(E_{10}\right)$.
\end{cor}

More interestingly, corollary \ref{cor:Epi Q-->classical} shows that
$\mathfrak{Q}_{p,q,d}\left(0,\dots,0\right)\left(\mathbb{R}\right)$ with $9\leq p+q \leq \vert \frac{1}{2}\Delta(E_8)$ 
maps onto $\mathfrak{g}\left(E_{8}\right)\left(\mathbb{R}\right)$
such that each generator is realized nontrivially and linearly independent:
\[
\phi\left(x_{i}\right)=f_{i}\,\forall\,i=1,\dots,8,\ \phi\left(x_{9}\right)=e_{\theta},\ \phi\left(x_{9+j}\right)=e_{\beta_j}.
\]
For $p+q=9$, it also sets apart one generator, as only $x_{9}$ is realized in
a positive root space. For affine root systems the the root system is  
\[
\Delta^{re}(A)=\left\{ n\delta+\alpha\ \vert\ n\in\mathbb{Z},\ \alpha\in\Delta\left(\mathring{A}\right)\right\} ,\ \Delta^{im}(A)=\left\{ n\delta\,\vert\,n\in\mathbb{Z}\setminus\{0\}\right\} ,
\]
where $\mathring{A}$ denotes the sub-diagram of finite type.
\begin{prop}
Let $\gamma=n\delta+\beta\in\Delta^{re}(E_{9})$ with $\beta\in\Delta\left(E_{8}\right)\setminus\left\{ -\theta,\alpha_{1},\dots,\alpha_{8}\right\} $
and $e_{\gamma}\in\mathfrak{g}(E_{9})_{\gamma}$, where $\theta$
denotes the highest $E_{8}$-root. Then $\left\langle e_{\gamma},f_{0},f_{1},\dots,f_{8}\right\rangle \cong\mathfrak{g}'\left(E_{9}\right)$
and the map 
\begin{equation}
x_{0}\mapsto f_{0},\,x_{1}\mapsto f_{1},\dots,x_{10}\mapsto e_{\gamma}\label{eq:hom Q_10 to g'(E9)}
\end{equation}
extends to an epimorphism of $\mathfrak{Q}_{10-d,d}\left(0,\dots,0\right)\rightarrow\mathfrak{g}'\left(E_{9}\right)$.
\end{prop}

\begin{proof}
The claim $\left\langle e_{\gamma},f_{0},f_{1},\dots,f_{8}\right\rangle \cong\mathfrak{g}'\left(E_{9}\right)$
follows from lemma \ref{lem:max support root implies derived subalgebra is generated}.
More directly one could also specialize to $\gamma=2\delta+\beta$
for some $\beta\in\Delta\left(E_{8}\right)$. By successive application
of $f_{1},\dots,f_{8}$ it is possible to reach the root space $2\delta-\theta$,
because in the finite root system $\Delta(E_{8})$, $-\theta$ is
a lowest root. But then 
\[
\left(2\delta-\theta\right)-2\alpha_{0}=\left(2\delta-\theta\right)-2\left(\delta-\theta\right)=\theta
\]
and from $e_{\theta}$ one obtains all $e_{1},\dots,e_{8}$. Since
the intermediate step reaches an element $e_{\delta}$ one can apply
$e_{-\theta}$ and obtain $e_{0}$ as well.

It is still left to show that (\ref{eq:hom Q_10 to g'(E9)}) defines
a homomorphism. The Serre-relations imply that $\text{ad}\left(f_{i}\right)^{2}\left(f_{j}\right)=0$
for all $i,j\in\{0,\dots,8\}$. For $\gamma=n\delta+\beta\in\Delta^{re}\left(E_{9}\right)$
one has for $i=1,\dots,8$ that $2\gamma-\alpha_{i},\gamma-2\alpha_{i}\notin\Delta\left(E_{9}\right)$
because $2\beta-\alpha_{i},\beta-2\alpha_{i}\notin\Delta\left(E_{8}\right)$.
Here it is actually crucial that $\beta\neq\alpha_{i}$ $\forall\,i=1,\dots,8$.
This implies that the corresponding root spaces are $0$-dimensional
and so one has
\[
\text{ad}\left(e_{\gamma}\right)^{2}\left(f_{i}\right)=0=\text{ad}\left(f_{i}\right)^{2}\left(e_{\gamma}\right)\ \forall\,i=1,\dots,8.
\]
Now for $f_{0}$ it is
\[
2\gamma-\alpha_{0}=2n\delta+2\beta-\left(\delta-\theta\right)=\left(2n-1\right)\delta+2\beta+\theta\notin\Delta\left(E_{9}\right),
\]
because $2\beta+\theta\notin\Delta\left(E_{8}\right)$ as $\beta\neq-\theta$.
For the same reason
\[
\gamma-2\alpha_{0}=(n-2)\delta+\beta+2\theta\notin\Delta\left(E_{9}\right),
\]
so that one also has 
\[
\text{ad}\left(e_{\gamma}\right)^{2}\left(f_{0}\right)=0=\text{ad}\left(f_{0}\right)^{2}\left(e_{\gamma}\right).
\]
This shows that that (\ref{eq:hom Q_10 to g'(E9)}) extends to a homomorphism.
As $\left\langle e_{\gamma},f_{0},f_{1},\dots,f_{8}\right\rangle \cong\mathfrak{g}'\left(E_{9}\right)$
it is an epimorphism.
\end{proof}
We believe that it is possible to find a similar epimorphism for $\mathfrak{Q}_{11-d,d}\left(0,\dots,0\right)\rightarrow\mathfrak{g}'\left(E_{10}\right)$. One simply has to find (and therefore show existence of) a real positive root with maximal support $\beta$, such that $\beta-2\alpha_i\notin\Delta\left(E_{10}\right)$ for all $i=1,\dots,10$.

\subsubsection*{A homomorphism $\mathfrak{Q}_{26-m,m}\left(0,\dots,0\right)(\mathbb{R})\twoheadrightarrow \mathfrak{g}\left(E_{10}\right)(\mathbb{R})$}
The following example shows that it is also possible to map QMSAs of higher rank to KM-algebras of much lower rank such that the generators are still linearly independent in the image. 
\begin{figure}[h]
\label{fig:E_10}
\begin{centering}
\includegraphics[scale=0.3]{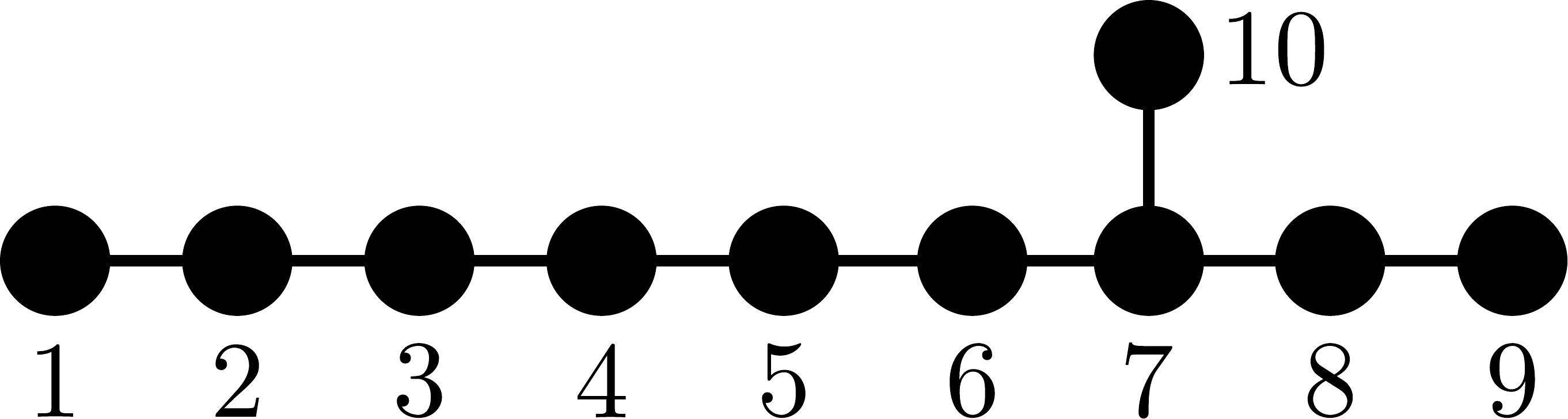}
\par\end{centering}
\caption{The generalized Dynkin diagram of type $E_{10}$.}
\end{figure}
Towards this, we make use of the structure of the $E_{10}$-root system's structure, expecially that it contains an $A_9$-sub-root system, in our notation spanned by the roots $\alpha_1,\dots,\alpha_9$. 
We map the first $10$ generators $x_1,\dots, x_{10}$ of $\mathfrak{Q}_{26-m,m}\left(0,\dots,0\right)(\mathbb{R})$ to the Chevalley generators $f_1,\dots,f_{10}$ as in the previous examples. 
But to avoid a detailed study of the $E_{10}$-root system we pick different roots $\beta_1,\dots,\beta_{16}\in \Delta^{re}_{+}\left(E_{10}\right)$ to accomplish $\left\langle f_1,\dots,f_{10},e_{\beta_1},\dots,e_{\beta_{16}} \right\rangle=\mathfrak{g}\left(E_{10}\right)$. 
Set $\beta=\alpha_7+\alpha_{10}$, where the labeling of simple roots is according to figure \ref{fig:E_10}. 
The product $\left(\beta\vert\gamma\right)$ with any positive root $\gamma\in\Delta\left(A_9\right)$ is completely detemined by the following ones:
\[
\left(\beta\vert\alpha_7\right)=1,\ \left(\beta_1\vert\ \alpha_6+\alpha_7\right)=0=\left(\beta\vert\ \alpha_7+\alpha_8\right),\ \left(\beta\vert\alpha_6+\alpha_7+\alpha_8\right)=-1,\ \left(\beta\vert\alpha_6\right)=-1=\left(\beta\vert+\alpha_8\right) .
\]
These determine everything, because if $\gamma=\sum_{i=1}^{9}k_i\alpha_i\in \Delta\left(A_9\right)$ is a positive $A_9$-root, then $k_i\in\{0,1\}$ such that $\text{supp}(\gamma)$ is connected. 
As the coefficients $k_i$ in  decomposition of any root $\delta=\sum_{i=1}^{10}k_i\alpha_i\Delta\left(E_{10}\right)$ have to be either all nonnegative or all nonpositive, one concludes that $\beta-\gamma\notin \Delta\left(E_{10}\right)$ for all $\alpha_7\neq\gamma\in \Delta\left(A_9\right)$. 
Hence one arrives at 
\begin{eqnarray*}
k\beta+\gamma,\,\beta+k\gamma\in\Delta\left(E_{10}\right) & \text{for }k\in\left\{ 0,1\right\}  & \text{if }\left(\beta\vert\gamma\right)=-1\\
k\beta+\gamma,\,\beta+k\gamma\in\Delta\left(E_{10}\right) & \text{for }k\in\left\{ 0\right\}  & \text{if }\left(\beta\vert\gamma\right)=0\\
k\beta+\gamma,\,\beta+k\gamma\in\Delta\left(E_{10}\right) & \text{for }k\in\left\{ 0,-1\right\}  & \text{if }\left(\beta\vert\gamma\right)=+1,
\end{eqnarray*}
where the last case only ocurrs for $\gamma=\alpha_7$. This shows with $\beta-\gamma\notin \Delta\left(E_{10}\right)$ that
\begin{equation}\label{eq:Serre-type relations in E_10 part 1}
\text{ad}\left(e_{\beta}\right)^2\left(e_\gamma\right)= 0 = \text{ad}\left(e_{\gamma}\right)^2\left(e_{\beta}\right)\ \forall\,\gamma\in\Delta\left(A_9\right)\setminus\left\{\alpha_1,\dots,\alpha_9\right\}\,.
\end{equation}
Note that while the exception of simple roots is not necessary for the above equality it is necessary for the following: 
\begin{equation}\label{eq:Serre-type relations in E_10 part 2}
\text{ad}\left(f_i\right)^2\left(e_\gamma\right)= 0 = \text{ad}\left(e_{\gamma}\right)^2\left(f_i\right)\ \forall\,\gamma\in\Delta\left(A_9\right)\setminus\left\{\alpha_1,\dots,\alpha_9\right\}\ \forall\,i=1,\dots,10.
\end{equation}
One now takes the $A_9$-roots
\[
\Gamma:=\bigcup_{i=1}^{8}\left\{ \alpha_i+\alpha_{i+1} \right\} \,\cup\,\bigcup_{i=1}^{7} \left\{ \alpha_i+\alpha_{i+1}+\alpha_{i+2} \right\}
\]
and observes that 
\[
\chi:=\left\{ f_1,\dots,f_{10}, e_\beta, e_\gamma\, \vert\, \gamma\in\Gamma\right\}
\]
is a generating set of $\mathfrak{g}\left(E_{10}\right)(\mathbb{R})$ that satisfies $\text{ad}(x)^2(y)=0$ $\forall\,x,y\in\chi$. As $\vert\chi\vert=26$, it provides an epimorphism $\mathfrak{Q}_{26-m,m}\left(0,\dots,0\right)(\mathbb{R})\twoheadrightarrow \mathfrak{g}\left(E_{10}\right)(\mathbb{R})$ for any signature and of course also if one replaces $\mathbb{R}$ with $\mathbb{C}$. The amount of arbitrary choices we made in this example show that these kinds of epimorphisms are not specific to $\mathfrak{Q}_{26-m,m}$ and $\mathfrak{g}\left(E_{10}\right)$ but we thought it is a potentially interesting example.

%%%%%%%%%%%%%%%%%%%%%%%%%%%%%

\end{document}